\documentclass[11pt]{amsart} % Defines the document class and options.
\makeatletter
\@ifundefined{subjclassname}{}{%
  
}
\makeatother
\usepackage{graphicx}
\usepackage[utf8]{inputenc} % Allows input of special characters (like accents).
\usepackage[T1]{fontenc} % Sets up font encoding for better output.

\usepackage{amsmath} % Essential for most mathematical formulas and environments (like aligned).
\usepackage{amssymb} % Provides a wide range of mathematical symbols (like \mathbb, \mathfrak).
\usepackage{amsthm} % Essential for defining and formatting theorem-like environments.
\usepackage{geometry} % Allows easy adjustment of margins.
\newtheorem{theorem}{Theorem}[section]

\newtheorem{proposition}[theorem]{Proposition}

\newtheorem{remark}[theorem]{Remark}
\newtheorem{lemma}[theorem]{Lemma}

\newtheorem{corollary}[theorem]{Corollary}

\title{Minimal surfaces over the Pitot quadrilaterals} % Title of the document
\author{Vladimir Dragovi\'c}
\address{Department of Mathematical Sciences,
The University of Texas at Dallas,
800 W Campbell Rd, Richardson, TX 75080, USA//Mathematical Institute SANU, Belgrade, Serbia}
\email{Vladimir.Dragovic@utdallas.edu}

\author{David Kalaj}
\address{University of Montenegro, Faculty of Natural Sciences and Mathematics,
Cetinjski put b.b., 81000 Podgorica, Montenegro}
\email{davidk@ucg.ac.me}

\date{\today}

\subjclass[2020]{58E12, 53A10, 31A05}

\keywords{Harmonic maps; Scherk surface; Gaussian curvature}
\begin{document}

\begin{abstract}
We develop a fully explicit framework for constructing Scherk-type minimal graphs over the Pitot quadrilaterals (i.e. such that the two pairs of opposite sides have the same total length).   For any Pitot quadrilateral  \(Q\), we first produce a harmonic diffeomorphism of the unit disk onto \(Q\), whose dilatation is the square of a M\"obius automorphism determined directly by the vertices of \(Q\).  Using this map as the  Weierstrass data, we obtain a minimal graph \(\Sigma\) whose Gauss map is a univalent M\"obius transformation and whose height function exhibits alternating blow-up behavior along opposite sides of \(Q\), mirroring the classical Scherk surfaces.  We further construct an associated canonical surface \(\Sigma^\diamond\), with the same boundary asymptotics, and prove a sharp curvature comparison theorem: at the harmonic center of \(Q\), among all bounded minimal graphs with matching normal direction and mixed derivative, \(\Sigma^\diamond\) uniquely maximizes the absolute Gaussian curvature.  This provides a complete and constructive description of Scherk-type minimal graphs over all, both convex or concave, Pitot quadrilaterals.
\end{abstract}

\maketitle % Command to display the title, author, and date

% Your main content starts here.
\section{Introduction}

In this paper, we investigate \emph{Scherk-type minimal graphs} defined over quadrilateral domains. A Scherk-type minimal graph means: A minimal surface given as a single-valued graph over a domain whose boundary data take alternating infinite values, producing a surface asymptotic to a set of vertical planes—generalizing the structure of Scherk’s second surface.
Let a quadrilateral be denoted by
\[
Q = Q(b_1,b_2,b_3,b_4),
\]
where the vertices \(b_i \in \mathbb{R}^2\) are ordered counterclockwise.
A necessary condition for the existence of a Scherk-type minimal graph over \(Q\) is that the sums of the lengths of the opposite sides are equal, namely
\begin{equation}\label{eq:side_condition}
|b_1b_2| + |b_3b_4| = |b_2b_3| + |b_4b_1|.
\end{equation}
Quadrilaterals that satisfy \eqref{eq:side_condition} are known as \emph{the Pitot quadrilaterals}. (Pitot proved that if $Q$ is a convex tangential quadrilateral, then its sides satisfy \eqref{eq:side_condition} in 1725; Durrande proved the converse in 1815, see \cite{Be}.) 

While the classical Scherk surfaces correspond to convex quadrilaterals satisfying~\eqref{eq:side_condition} (see \cite{JENKINSSERRIN} and their application to the solution of Gaussian curvature conjecture \cite{kalaj2025gaussian}),
our analysis shows that a \emph{Scherk-type minimal graph can be constructed above every Pitot quadrilateral}, including concave ones.
Thus, the convexity of \(Q\) is not a necessary condition for the existence of such minimal graphs.

The resulting minimal surface
\[
\mathcal{T} : \mathbb{D} \to \mathbb{R}^3,\quad \mathbb{D}=\{z||z|<1\},
\]
has \emph{the Gaussian unit normal} \(N(z)\) lying entirely in the upper hemisphere.
This construction provides a natural geometric extension of classical Scherk surfaces and reveals a direct connection between the shape of \(Q\) and the analytic form of the corresponding minimal graph.

A related result was previously obtained by Bshouty and Hengartner~\cite{BshoutyHengartner1997},
but our approach refines and extends their framework by providing an \emph{explicit formula} for the unit normal \(N(z)\) expressed directly in terms of the vertices of \(Q\).
This explicit dependence clarifies the geometric structure and facilitates analytic study.

Moreover, the constructed minimal graphs satisfy the characteristic \emph{Scherk-type asymptotic property}:
if \(z \to \pm 1\) or \(z \to \pm e^{\imath p}\), then the third coordinate of the immersion \(T(z)\) tends to \(\pm\infty\).
This mirrors the asymptotic behavior of classical Scherk surfaces and confirms that Scherk’s construction naturally extends to all, {\it both convex and concave}, Pitot quadrilateral domains. Let us observe that a Pitot quadrilateral cannot be self-intersecting, except in degenerate cases where two vertices coincide. 

Our results therefore unify and broaden the classical theory of minimal surfaces, linking boundary geometry with analytic representation.
They demonstrate that the class of domains supporting Scherk-type minimal graphs coincides precisely with the class of Pitot quadrilaterals.

The next three theorems constitute the core of our results and provide a complete, fully explicit description of Scherk-type minimal graphs over Pitot quadrilaterals.  Theorem~\ref{teo1} establishes the existence of a harmonic diffeomorphism from the unit disk onto any given Pitot quadrilateral~\(Q\), with the dilatation determined by a M\"obius automorphism whose parameters depend directly and explicitly on the vertices of~\(Q\).  Building on this, Theorem~\ref{teo2} shows that these harmonic maps generate minimal graphs over~\(Q\) whose Weierstrass data---and in particular the Gauss map---can be written in a closed form and exhibit the characteristic alternating blow-up behavior of Scherk-type surfaces along opposite sides of the quadrilateral.  Finally, Theorem~\ref{etreta} yields a sharp curvature comparison: among all bounded minimal graphs over~\(Q\) sharing the same normal direction and mixed derivative at the harmonic center, the canonical surface~\(\Sigma^\diamond\) constructed in Theorem~\ref{teo2} has a strictly maximal absolute Gaussian curvature at this point.

\begin{theorem}\label{teo1}
Let \(Q\) be a Pitot quadrilateral with consecutive vertices \(b_1,b_2,b_3,b_4\).
Then there exists a parameter \(p=p(b_1,b_2,b_3,b_4)\in(0,\pi)\), such that the step function \(F:[0,2\pi)\to \mathbb{C}\), given by
\[
F(\theta)=
\begin{cases}
b_0=b_4, & \theta \in [0,\,p),\\[2pt]
b_1, & \theta \in [p,\,\pi),\\[2pt]
b_2, & \theta \in [\pi,\,\pi+p),\\[2pt]
b_3, & \theta \in [\pi+p,\,2\pi),
\end{cases}
\]
has the Poisson extension \(f=P[F]\)  to the unit disk \(\mathbb{D}\),
\begin{equation}\label{eq:repres}
f(z)=h(z)+\overline{g(z)},\qquad z\in\mathbb{D},
\end{equation}
that is a sense-preserving harmonic mapping with the following properties:

\begin{enumerate}
\item [(i)] The dilatation \(\omega=g'/h'\) satisfies
\[
\omega(z)=\phi(z)^2,
\]
where \(\phi\) is a Möbius automorphism of \(\mathbb{D}\), 
$$\phi(z) = X\frac{z-z_\circ}{1-z\bar{z_\circ}},$$ 
with $|X|=1$ and $|z_\circ|<1$ being explicitly given in the hyperbolic coordinates with respect to the hyperbola containing the points $b_2$ and $b_4$ and with the foci at the points $b_1$ and $b_3$.

\item [(ii)] The function $f$ is a harmonic diffeomorphism of the unit disk onto $Q$.
\end{enumerate}
\end{theorem}

\begin{theorem}\label{teo2}
Let $Q=Q(b_1,b_2,b_3,b_4)$ be a Pitot quadrilateral.  Then there exists a minimal graph $\Sigma$ defined over the quadrilateral $Q$ 
given in the Enneper–Weierstrass form by
\[
\Sigma \;=\; \bigl\{\,(\Re f(z),\,\Im f(z),\,T(z)) : z\in\mathbb{D}\,\bigr\},\qquad
T(z) \;=\; 2\,\Im \int_0^z {\bf p}(\zeta)\,{\bf q}(\zeta)\,d\zeta,
\]
whose unit normal at the point $(f(z),T(z))$ is ${\bf q}(z)$. The Weierstrass data $({\bf p,q})$ can be chosen so that ${\bf q}$ is a Möbius automorphism of the unit disk and for $f$ from \eqref{eq:repres} ${\bf q}=\sqrt{g'/h'}$ and ${\bf p}=h'$. In particular, the Gauss map is a univalent diffeomorphism of the surface $\Sigma$ onto the open upper hemisphere of $\mathbb{S}^2$.

The surface $\Sigma^\diamond=\{(u,v, \mathbf{f}^\diamond(u,v)): (u,v)\in Q\}$, where $\mathbf{f}^\diamond=T(f^{-1})$, has the following Scherk-type asymptotics:
\[
\text{if } (u,v)\to \zeta\in (b_1,b_2)\cup(b_3,b_4) \text{ then } \mathbf{f}(u,v)\to +\infty,\]
\[ \text{if } (u,v)\to\zeta\in (b_4,b_1)\cup(b_2,b_3) \text{ then } \mathbf{f}(u,v)\to -\infty.
\]
\end{theorem}

\begin{remark}
 Some of the results in Theorem \ref{teo1} and Theorem \ref{teo2}—with the exception of the explicit construction of the constants—were previously obtained by Bshouty and Hengartner \cite{BshoutyHengartner1997}; see also Laugesen’s paper  \cite{Laugesen1997} for related results. However, our approach is more direct and fully constructive, and the resulting proofs differ substantially from theirs.

However, the  construction of the Enneper–Weierstrass parameters that explicitly depend on the vertices of a concave quadrilateral is a challenging problem, and its establishment has been essential to prove the second part of Theorem~\ref{teo2}.
\end{remark}

If a harmonic diffeomorphism $f:\mathbb D \to Q$ is given, then we call the unique point $c_0=f(0)\in Q$ \emph{the harmonic center} of $Q$.

Finally, we prove the following theorem 

\begin{theorem}\label{etreta}
Given a Pitot quadrilateral $Q$, assume that $\Sigma=\{(u,v, \mathbf{f}(u,v)): (u,v)\in Q\}$ is any bounded minimal graph above it. Then for $\xi = (c_0, \mathbf{f}(c_0))$ we have $$\left|\mathcal{K}(\xi)\right| <  \left|\mathcal{K}^\diamond(\xi)\right|,$$ provided that the unit normals of $\Sigma$ and $\Sigma^\diamond$ at $\xi$ coincide and the mixed derivatives of $\mathbf{f}$ and $\mathbf{f}^\diamond$ coincide at the point $c_0$, where $c_0$ is the harmonic center of $Q$  and $\mathcal {K}$ and $\mathcal {K}^\diamond$ are the Gaussian curvatures of the surfaces $\Sigma$ and $\Sigma^\diamond$ respectively. The result is sharp.
\end{theorem} 
\begin{remark}
In Theorem~\ref{kater}---a reformulation of Theorem~\ref{etreta}---we derive an explicit formula for the Gaussian curvature of $\Sigma^\circ$ at the point $\xi$, expressed solely in terms of the vertices of the polygon $Q$.
\end{remark}
\section{Preliminaries}
Let $\Omega$ be a polygon with $n$ distinct vertices
$b_1, b_2, \ldots, b_n$ taken in counterclockwise order along its boundary.
Choose a partition
$0 = t_0 < t_1 < \cdots < t_n = 2\pi$
of the interval $[0, 2\pi]$ and define the step function
\[
\varphi(e^{it}) = b_k, \qquad t_{k-1} < t < t_k, \quad k = 1, 2, \ldots, n.
\]
The harmonic extension of $\varphi$ to the unit disk is given by the Poisson integral
\[
f(z) = \frac{1}{2\pi} \int_0^{2\pi}
\frac{1 - |z|^2}{|e^{it} - z|^2}\,\varphi(e^{it})\,dt.
\]
By the {Radó–Kneser–Choquet theorem} (see e.g. \cite{sheil, Duren2004}), if the boundary function
$\varphi$ parametrizes a \emph{convex} Jordan curve, then $f$ is
\emph{univalent} in the unit disk $\mathbb{D}$ and maps $\mathbb{D}$ onto $\Omega$.
As a corollary, by using the limiting procedure, the above formula defines a one-to-one harmonic mapping
$f:\mathbb{D}\to\Omega$ whenever $\Omega$ is convex. If the quadrilateral is not convex, then we do not have any more  such a nice criterion  for univalence of the Poisson extension of the step function.

The mapping $f$ admits a canonical decomposition $f = h + \bar{g}$,
where $h$ and $g$ are analytic in $\mathbb{D}$ and $g(0)=0$.
Writing the Poisson kernel in the complex form,
\[
\frac{1 - |z|^2}{|e^{it} - z|^2}
   = \frac{e^{it}}{e^{it} - z}
     + \frac{\bar{z}}{e^{-it} - \bar{z}},
\]
we obtain
\[
h(z) = \frac{1}{2\pi}\int_0^{2\pi}
        \frac{e^{it}}{e^{it} - z}\,\varphi(e^{it})\,dt,
\qquad
g(z) = \frac{1}{2\pi}\int_0^{2\pi}
        \frac{\bar{z}}{e^{-it} - \bar{z}}\,\varphi(e^{it})\,dt.
\]

Differentiating and substituting the step function $\varphi(e^{it})$
yields
\[
h'(z) = \frac{1}{2\pi \imath}\sum_{k=1}^n
         b_k\!\left(\frac{1}{z - \zeta_k}
                        - \frac{1}{z - \zeta_{k-1}}\right),
\qquad
g'(z) = \frac{1}{2\pi \imath}\sum_{k=1}^n
         \overline{b_k}\!\left(\frac{1}{z - \zeta_k}
                        - \frac{1}{z - \zeta_{k-1}}\right),
\]
where $\zeta_k = e^{it_k}$ and the indices are taken modulo $n$.
Equivalently,
\[
h'(z) = \sum_{k=1}^n \frac{c_k}{z - \zeta_k},
\qquad
g'(z) = -\sum_{k=1}^n \frac{\overline{c_k}}{z - \zeta_k},
\quad
c_k = \frac{1}{2\pi \imath}\,(b_k - b_{k+1}),
\quad b_{n+1} = b_1.
\] For the above facts we refer to Duren's book  \cite{Duren2004} and Sheil-Small's paper \cite{sheil}.

For the special configuration
$t_1 = 0$, $t_2 = p$, $t_3 = \pi$, $t_4 = \pi + p$, and $t_5 = 2\pi$,
these formulas simplify accordingly. We assume that $b_1,b_2,b_3,b_4$ are the vertices of a quadrilateral ordered in clock-wise direction with respect to the boundary.

Then,
\begin{equation}\label{haprim}
h'(z) = \frac{i}{2\pi} \left(
  \frac{b_2-b_3}{e^{i p}-z}
  + \frac{-(1+z) b_1 + (1+z) b_2 - (-1+z)(b_3-b_4)}{-1+z^2}
  + \frac{b_1-b_4}{e^{i p}+z}
\right)
\end{equation}

and \begin{equation}\label{geprim}
g'(z)
= \frac{i}{2\pi}\left(
  \frac{\overline{b_1} - \overline{b_2}}{1 - z}
+ \frac{\overline{b_2} - \overline{b_3}}{e^{i p} - z}
+ \frac{\overline{b_4} - \overline{b_3}}{1 + z}
+ \frac{\overline{b_1} - \overline{b_4}}{e^{i p} + z}
\right).
\end{equation}

Then there are two quadratic polynomials $P_2$ and $Q_2$ so that  $$\frac{g'(z)}{h'(z)}=\frac{P_2(z)}{Q_2(z)},$$  where
\[
P_2(z)=\frac{\imath}{2\pi} \Bigg(
\begin{aligned}& (1 + e^{\imath p})(e^{\imath p} - z)(1 + z)\,\overline{b_1}
- (-1 + e^{\imath p})(1 + z)(e^{\imath p} + z)\,\overline{b_2} \\
& \quad + (-1 + z)\, e^{2 \imath p} (\overline{b_3} - \overline{b_4})
+ (-1 + z)\, z (\overline{b_3} - \overline{b_4}) \\
& \quad + (-1 + z)\, e^{\imath p}(1 + z)(\overline{b_3} + \overline{b_4})
\end{aligned}
\Bigg)
\]

\[
Q_2(z)=\frac{\imath}{2\pi} \Bigg(
\begin{aligned}
& e^{2 \imath p} \Big( (1 + z) b_1 - (1 + z) b_2 + (-1 + z)(b_3 - b_4) \Big) \\
& \quad + z \Big( -(1 + z) b_1 + (1 + z) b_2 + (-1 + z)(b_3 - b_4) \Big) \\
& \quad - e^{\imath p} (-1 + z^2) (b_1 + b_2 - b_3 - b_4)
\end{aligned}
\Bigg).
\]
We need to find \(p\) such that the quadratic polynomials $Q_2$ and $P_2$ are  perfect squares:
As $Q_2'(z')=0$ for
\[
z' =
\frac{(-1 + e^{2 \imath p}) \,(b_1 - b_2 + b_3 - b_4)}
     {2 \Big( b_1- b_2- b_3 + b_4+ e^{\imath p} \left(b_1  + b_2  -  b_3  -  b_4\right) \Big)},
\] we need to find \(p\) such that \(Q_2(z') = 0\).

Such $p$ satisfies the relation  \[
\begin{aligned}
Y:=& 3 b_1^2 + 3 b_2^2 + 3 b_3^2 + 3 b_4^2
+ 2 b_2 (b_3 - 5 b_4) + 2 b_3 b_4
+ 2 b_1 (b_2 - 5 b_3 + b_4) \\
& \quad + 4 (b_1 + b_2 - b_3 - b_4)(b_1 - b_2 - b_3 + b_4) \cos p \\
& \quad + (b_1 - b_2 + b_3 - b_4)^2 \cos 2p=0
\end{aligned}
.\] Let $b_j = u_j + \imath v_j$, then $\mathrm{Im}(Y)=0$ if and only if $$A\cos (2p)+B \cos p+C =0,$$ where

\[
{
\begin{aligned}
A &= 2\,(u_1 - u_2 + u_3 - u_4)(v_1 - v_2 + v_3 - v_4), \\[6pt]
B &= 4\Big[(u_1 + u_2 - u_3 - u_4)(v_1 - v_2 - v_3 + v_4)
+ (u_1 - u_2 - u_3 + u_4)(v_1 + v_2 - v_3 - v_4)\Big], \\[6pt]
C &= 6(u_1v_1+u_2v_2+u_3v_3+u_4v_4)
+2[(u_2v_3+u_3v_2)-5(u_2v_4+u_4v_2)]
+2(u_3v_4+u_4v_3) \\
&\quad +2[(u_1v_2+u_2v_1)-5(u_1v_3+u_3v_1)+(u_1v_4+u_4v_1)].
\end{aligned}
}
\]
 Thus,
\[
\text{Im}(Y)=0 \quad \text{which is equivalent to} \quad A\cos(2p)+B\cos p+C=0.
\]
Then, a  direct computation yields to  $A+B+C=16 (u_1-u_3)(v_1-v_3)$. Thus, \[
A\cos(2p) + B\cos p + C = 0.
\]

Using \(\cos(2p) = 2\cos^2 p - 1\) and letting \(c = \cos p\), we get
\[
A(2c^2 - 1) + Bc + C = 0,
\quad \text{which implies} \quad
2A c^2 + Bc + (C - A) = 0.
\]
We will prove later that we can take $C = -A - B$, i.e. $v_1=v_3=0$.  Then, this becomes
\[
2A c^2 + Bc - (2A + B) = 0,
\quad \text{which implies} \quad
2A(c^2 - 1) + B(c - 1) = 0.
\]
Factoring the left hand side, we get
\[
(c - 1)\,[\,2A(c + 1) + B\,] = 0.
\]
Hence,
\[
{
\cos p = 1
\quad \text{or} \quad
\cos p = -1 - \frac{B}{2A}.
}
\]
Thus,
\begin{equation}\label{ppe}
{
p = 2n\pi
\quad \text{or} \quad
p = \arccos\!\left(-1 - \frac{B}{2A}\right),
\quad n \in \mathbb{Z},
}
\end{equation}
provided that \(-1 - \dfrac{B}{2A} \in [-1,1].\)

Assuming $p\neq 2n\pi$, because otherwise we have a degenerate case, we conclude that $p = \arccos\!\left(-1 - \frac{B}{2A}\right)$.
Then \begin{equation}\label{cop}
\cos p =E,
\end{equation} 
where 
\begin{equation}\label{EEE}
E=\arccos\!\left(-1 - \frac{B}{2A}\right).
\end{equation} 
We will prove later that $E\in[-1,1]$.
Similarly, $P_2'(z_\circ)=0$ for  
\[
z_\circ =
\frac{(-1 + e^{2 \imath p}) \,(\overline{b_1} - \overline{b_2} + \overline{b_3} - \overline{b_4})}
     {2\Big[(\overline{b_1} - \overline{b_2} - \overline{b_3} + \overline{b_4})
     + e^{\imath p}(\overline{b_1} + \overline{b_2} - \overline{b_3} - \overline{b_4})\Big]}.
\]

The same $p$ satisfying \eqref{ppe} will satisfy the relation $P_2(z_\circ)=0$, which is crucial.

Since \begin{equation}\label{mu0}\mu(0)=
\frac{P_2(0)}{Q_2(0)} =
\frac{(e^{\imath p}+1)\,(\overline{b_1}-\overline{b_3}) + (1-e^{\imath p})\,(\overline{b_2}-\overline{b_4})}
     {(e^{\imath p}+1)\,(b_1-b_3) + (1-e^{\imath p})\,(b_2-b_4)},
\end{equation}
we obtain that
\[
\mu(\zeta)=\frac{g'(\zeta)}{h'(\zeta)}=X\,\frac{(\zeta-z_\circ)^2}{\bigl(1-\zeta\,\overline{z_\circ}\bigr)^2},
\]
where  
$$X=
\frac{
4\Big[(e^{\imath p}+1)(\overline{b_1}-\overline{b_3})
      +(1-e^{\imath p})(\overline{b_2}-\overline{b_4})\Big]
     \Big[(e^{\imath p}+1)(\overline{b_1}-\overline{b_3})
      +(e^{\imath p}-1)(\overline{b_2}-\overline{b_4})\Big]^{2}
}
{
\Big(b_1+b_2-b_3-b_4 + (b_1-b_2-b_3+b_4)e^{\imath p}\Big)
(-1+e^{2\imath p})^{2}
(\overline{b_1}-\overline{b_2}+\overline{b_3}-\overline{b_4})^{2}
}
$$

\section{Reduction of the problem to two complex variables $z,w$}

In the accordance with \eqref{eq:side_condition}, we are given four complex numbers $b_1,b_2,b_3,b_4$ satisfying the Pitot condition
\begin{equation}\label{b1234}
|b_1-b_2|+|b_3-b_4|=|b_1-b_4|+|b_2-b_3|.
\end{equation}
Since this relation is invariant under similarities of the complex plane
(translation, rotation, and homothety),
we may place $b_1$ and $b_3$ at the convenient symmetric positions
$-1$ and $1$ on the real axis.

Let
\[
T_1(u)=u-\frac{b_1+b_3}{2}, \qquad
T_2(u)=\frac{2u}{\,b_3-b_1\,}.
\]
Set \(T = T_2 \circ T_1\). Then,
\[
T(b_1) = -1, \qquad
T(b_2) =: z, \qquad
T(b_3) = 1, \qquad
T(b_4) =: w.
\]
Thus the similarity \(T\) maps the four points \((b_1,b_2,b_3,b_4)\) to \((-1,z,1,w)\).

The Pitot condition \eqref{b1234} becomes
\[
|z+1|-|z-1|=|w+1|-|w-1|,
\]
which shows that $z$ and $w$ lie on the same  hyperbola with the foci $\pm 1$.

Thus, we define the new variables:
let
  \[
w = u+\imath v=\frac{2b_4 - b_1 - b_3}{b_3 - b_1}=\sin m\cosh s+\imath\,\cos m\sinh s, \]
\[z =x+\imath y=  \frac{2b_2 - b_1 - b_3}{b_3 - b_1}=\sin m\cosh t+\imath\,\cos m\sinh t.
\]
Then, equality \eqref{b1234} takes the normalized canonical form
\begin{equation}\label{hyperbola}
{|w + 1| - |w-1| = |z + 1| - |z-1|.}
\end{equation}
From \[
z_\circ =
\frac{(-1 + e^{2 \imath p}) \,(\overline{b_1} - \overline{b_2} + \overline{b_3} - \overline{b_4})}
     {2\Big[(\overline{b_1} - \overline{b_2} - \overline{b_3} + \overline{b_4})
     + e^{\imath p}(\overline{b_1} + \overline{b_2} - \overline{b_3} - \overline{b_4})\Big]},
\]
substituting
\[
\overline{b_1}=-1, \qquad
\overline{b_2}=x-\imath y, \qquad
\overline{b_3}=-1, \qquad
\overline{b_4}=u-\imath v,
\]
we obtain
\begin{equation}\label{zcir}
z_\circ
=
\frac{i e^{i p}\sin p\,\bigl(-(x+u) + i(y+v)\bigr)}
     {(u - x - 2 + i(y - v))
     + e^{i p}\bigl(x - u - 2 + i(v - y)\bigr)},
\end{equation}
and

\[
X =
\frac{
4\Big[-2(e^{\imath p}+1) + (1-e^{\imath p})(\bar z-\bar w)\Big]
 \Big[-2(e^{\imath p}+1) + (e^{\imath p}-1)(\bar z-\bar w)\Big]^{2}
}
{
\big(z-w-2 + (w-z-2)e^{\imath p}\big)\,
(-1+e^{2\imath p})^{2}\,
(\bar z+\bar w)^{2}
}.
\]
Then, 
\begin{equation}\label{XXX}
X=-\frac{
4 e^{- i p}\,\csc^{2}(p)\,
\bigl( 2\cos(\tfrac{p}{2}) + (i u + v - i x - y)\,\sin(\tfrac{p}{2}) \bigr)^{2}\,
\bigl( 2\cos(\tfrac{p}{2}) + (-i u - v + i x + y)\,\sin(\tfrac{p}{2}) \bigr)
}{
\bigl(u - i(v + i x + y)\bigr)^{2}\,
\bigl( 2\cos(\tfrac{p}{2}) + (-i u + v + i x - y)\,\sin(\tfrac{p}{2}) \bigr)
}.
\end{equation}

We will use the equation \eqref{hyperbola} and the fact that $z$ and $w$ belong to the same branch of the hyperbola $|\zeta+1|-|\zeta-1|=2\kappa$, $\kappa\in[-1,1]$. Namely, we use the hyperbolic coordinates \begin{equation}\label{eq:h}
h(t)=a\cosh t+\imath\,b\sinh t,\qquad a=\kappa,\quad b=\sqrt{1-\kappa^2}.
\end{equation}

Observe that \eqref{eq:h} is equivalent to the equation  
\[
\frac{x^{2}}{\sin^{2} m}-\frac{y^{2}}{\cos^{2} m}=1,
\]
where $x+i y = h(t)$, and $\kappa=\sin m$. 
\subsection{The geometric meaning of the parameter $t$}
\label{sec:geom_t}

Consider the hyperbola
\[
\frac{x^{2}}{\sin^{2} m}-\frac{y^{2}}{\cos^{2} m}=1,
\]
which admits the standard Lorentzian parametrization
\[
x(t)=\sin m\,\cosh t, \qquad
y(t)=\cos m\,\sinh t .
\]
The parameter $t$ has a natural geometric interpretation in the
two--dimensional Minkowski plane.

\medskip
\noindent\textbf{1. Hyperbolic angle (rapidity).}
The parameter $t$ is the \emph{hyperbolic angle} of the point
$(x(t),y(t))$ with respect to the $x$--axis.
At $t=0$ the point is at the vertex $(\sin m,0)$. By increasing $t$, one
moves the point along the right branch of the hyperbola.

\medskip
\noindent\textbf{2. Lorentzian arc--length.}
The Minkowski metric is
\[
ds^{2}=dx^{2}-dy^{2}.
\]
Differentiating the parametrization, gives
\[
dx = \sin m\,\sinh t\,dt, \qquad
dy = \cos m\,\cosh t\,dt,
\]
and hence
\[
ds^{2} = dx^{2}-dy^{2}
        = -dt^{2}.
\]
Thus,
\[
|ds| = dt,
\]
so the parameter $t$ measures the arc--length of the hyperbola with
respect to the Lorentzian metric.  Equivalently, $t$ is the intrinsic
Minkowski distance of the point $z=x+i y$ from the vertex $(\pm \sin m, 0)$ along the hyperbola.

\begin{lemma}\label{lepo}
For every real $x,u,y$, and $v$, such that $z=u+iv$ and $w=x+\imath y$ satisfy $|z+1|-|z-1|=|w+1|-|w-1|$, define $E$ by \eqref{EEE}. Then,
\[
E = \frac{u v-3 v x-3 u y+x y}{(u+x) (v+y)},
\]

and $|E|<1$.
\end{lemma}

\begin{proof}
We use the parametrization
\begin{equation}\label{eq:h}
h(\tau)=\sin m \cosh \tau+\imath\,\cos m\sinh \tau ,\qquad m\in[0,\pi],
\end{equation}
and set $z=h(s)=u+iv$ and $w=h(t)=x+iy$. From \eqref{eq:h}, we get
\begin{equation}\label{eq:uxvy}
u=\sin m\cosh s,\quad v=\cos m\sinh s,\qquad
x=\sin m\cosh t,\quad y=\cos m\sinh t.
\end{equation}
Then,
\begin{equation}\label{eq:Eraw}
E=E(s-t)=1-\frac{4}{1+\cosh (s-t)}.
\end{equation}

Thus: $-1\le E \le 1$ for all $s\ne t$; $E(0)=1$ in the degenerate case $s=t$. \end{proof}

\begin{lemma}
For $X$ defined in \eqref{XXX}, we have $|X|=1$.
\end{lemma}

\begin{proof}
After a straightforward calculation, by using \eqref{eq:uxvy}, we obtain  $$X=\frac{\left(\imath e^{s/2}+e^{t/2}\right)^2 \left(1+e^{\imath m+\frac{s+t}{2}}\right)^2}{\left(e^{s/2}+\imath e^{t/2}\right)^2 \left(e^{\imath m}+e^{\frac{s+t}{2}}\right)^2}.$$
Let  $s,t,m \in \mathbb{R}.$
 Set  $a = e^{s/2} > 0$,\; $b = e^{t/2} > 0,\;$ $c = e^{(s+t)/2} > 0$.

Then
\[
X = \frac{(\imath a + b)^2 (1 + c e^{\imath m})^2}{(a + \imath b)^2 (e^{\imath m} + c)^2}.
\]

We compute the moduli term by term and see that each fraction has modulus $1$, thus:
\[
|X|
= \left| \frac{\imath a + b}{a + \imath b} \right|^2
  \left| \frac{1 + c e^{\imath m}}{e^{\imath m} + c} \right|^2
= 1.
\]
\end{proof}

\begin{lemma}\label{lem:z0}
Let
\[
z_\circ =
\frac{i e^{i p}\sin p\,\bigl(-(x+u) + i(y+v)\bigr)}
     {(u - x - 2 + i(y - v))
     + e^{i p}\bigl(x - u - 2 + i(v - y)\bigr)},
\]
where
\[
p = \arccos\!\frac{u v-3 v x-3 u y+x y}{(u+x) (v+y)},
\qquad
v = \frac{\sqrt{(1-\kappa^2)\bigl((1+2u)^2 - \kappa^2\bigr)}}{2\kappa},
\]
and
\[
\kappa = \sqrt{(1+x)^2 + y^2} - \sqrt{(x-1)^2 + y^2}, \qquad \kappa \in (-1,1).
\]
Then \( |z_\circ| < 1 \).
\end{lemma}

\begin{proof}
Let \( 2\kappa = |z+1| - |z-1| = |w+1| - |w-1| \).
Use the parametrization
\[
h(\tau) =  a \cosh\tau + \imath\,{b}\sinh\tau,
\qquad a = \kappa, \quad b = \sqrt{1 - \kappa^2}.
\]
Set \( z = h(s) = u + \imath v \) and \( w = h(t) = x + \imath y \).

Introduce the auxiliary parameters
\[
a = \sin m, \qquad \rho = e^{\,s - t} > 0,
\]
and define
\[
p = \arccos\!\left( \frac{\rho^2 - 6\rho + 1}{(\rho + 1)^2} \right).
\]
Then,
\[
u =a \cosh s, \quad
v = b \sinh s, \quad
x = {a}\cosh t, \quad
y = {b}\sinh t.
\]

Substituting the last formulas into the definition of \(z_\circ\), gives for $m\in[0,\pi/2]$ that 
\[
z_\circ =
\frac{\left(e^{s/2}+i e^{t/2}\right) \left(-1+e^{\frac{1}{2} (2 i m+s+t)}\right)}{\left(e^{s/2}-i e^{t/2}\right) \left(1+e^{\frac{1}{2} (2 i m+s+t)}\right)}.
\]

For $m\in[\pi/2,\pi]$, we have 
$$z_\circ=\frac{\left(e^{s/2}+i e^{t/2}\right) \left(e^{i m}+e^{\frac{s+t}{2}}\right)}{\left(e^{s/2}-i e^{t/2}\right) \left(-e^{i m}+e^{\frac{s+t}{2}}\right)}.$$
Then, we obtain that 
\[
|z_\circ|
=
\frac{
\cosh\!\left(\tfrac{s+t}{2}\right) - \lvert \cos m \rvert
}{
\cosh\!\left(\tfrac{s+t}{2}\right) + \lvert \cos m \rvert
},
\] 
thus  \( |z_\circ| < 1 \).
\end{proof}

\section{Proofs of Theorem~\ref{teo1} and Theorem~\ref{teo2}}

We already proved that $\mu_f(z) =\frac{g'(z)}{h'(z)}= ({\bf q}(z))^2,$ where ${\bf q}:\mathbb{D}\to \mathbb{D}$ is a M\"obius transformation of the unit disk, and in particular, $h'$ does not vanish in $\mathbb{D}$. Then, by the Sheil-Small theorem \cite{sheil}, the mapping $f$ is a diffeomorphism.
We know that 
$$\Sigma= \{(\Re f, \Im f, 2\Im \left(\int^z_0 {\bf p}(z) {\bf q}(z)dz\right) ): z\in \mathbb{D}\},$$ 
and we prove the Scherk-type property of the minimal graph $\Sigma$. 

We begin by defining the derivatives
\begin{equation}\label{AB}
g'(z)=A\,\frac{(z-z_\circ)^2}{(1-z^2)(e^{2\imath p}-z^2)},
\qquad
h'(z)=B\,\frac{(1-z\bar z_\circ)^2}{(1-z^2)(e^{2\imath p}-z^2)},
\end{equation}
where \(A\neq0\) and \(p\in\mathbb{R}\).
Multiplication of these two functions gives
\[
g'(z)h'(z)
=AB\,\frac{(z-z_\circ)^2(1-z\bar z_\circ)^2}{(1-z^2)^2(e^{2\imath p}-z^2)^2}.
\]

By \eqref{mu0}, and the relation $e^{\imath p}=\frac{(\imath+e^{j})^{2}}{(-\imath+e^{j})^{2}}$, with $j=(s-t)/2$, in view of  \eqref{eq:uxvy} and \eqref{AB}  we have
\[\begin{split}
\frac{g'(0)}{h'(0)}&=\frac{A z_\circ^2}{B}=\frac{(e^{\imath p}+1)\,(\overline{b_1}-\overline{b_3}) + (1-e^{\imath p})\,(\overline{b_2}-\overline{b_4})}
     {(e^{\imath p}+1)\,(b_1-b_3) + (1-e^{\imath p})\,(b_2-b_4)}
     \\&= \frac{-1-e^{\imath p}+\left(-1+e^{\imath p}\right) (u-\imath v-x+\imath y)}{-1-e^{\imath p}+\left(-1+e^{\imath p}\right) (u+\imath (v+\imath x-y))}
      \\&= -\left(\frac{-1+e^{\frac{1}{2} (2 \imath m+s+t)}}{e^{\imath m}+e^{\frac{s+t}{2}}}\right)^{\pm 2},\end{split}\] with $+2$ if $\cos m\ge 0$ and $-2$ if $\cos m<0$. From now on, we assume, without loss of generality,  that $m\in[0,\pi/2]$.
Next, we introduce a scaling factor \(Z\) that relates the parameters \(A\) and \(B\) as follows:
\[
Z=\frac{\left(\imath e^{s/2}+e^{t/2}\right)^2 \left(1+e^{\imath m+\frac{s+t}{2}}\right)^2}
{\left(e^{s/2}+\imath e^{t/2}\right)^2 \left(e^{\imath m}+e^{\frac{s+t}{2}}\right)^2},
\qquad
A=BZ.
\]
To determine \(B\), we evaluate \(h'(z)\) at \(z=0\).
A straightforward computation gives from \eqref{haprim} that
\[
h'(0)
=\frac{-2\imath+\dfrac{-2-\sin(m+\imath s)+\sin(m+\imath t)}{\imath+\sinh\!\left(\tfrac{s-t}{2}\right)}}{\pi}
=Be^{-2\imath p},
\]
and therefore
\[
B=e^{2\imath p}\,
\frac{-2\imath+\dfrac{-2-\sin(m+\imath s)+\sin(m+\imath t)}{\imath+\sinh\!\left(\tfrac{s-t}{2}\right)}}{\pi}.
\]

Using these expressions for \(A\) and \(B\), we can now write
\[
K(z)=\sqrt{g'(z)h'(z)}
=BZ^{1/2}\,\frac{(z-z_\circ)(1-z\bar z_\circ)}{(1-z^2)(e^{2\imath p}-z^2)}
=C\,\frac{(z-z_\circ)(1-z\bar z_\circ)}{(1-z^2)(e^{2\imath p}-z^2)},
\]
where the constant \(C = B Z^{1/2}\).

%To proceed, it is convenient to express \(z_\circ\) explicitly in terms of the parameters \(j, k, m\):
%\[
%z_\circ=\frac{(\imath+e^{j})(-1+e^{k+\imath m})}{(-\imath+e^{j})(1+e^{k+\imath m})}.
%\]
%It follows directly that
%\[
%|z_\circ|=\sqrt{1-\frac{2\cos m}{\cos m+\cosh k}}<1,
%\]
%which ensures that \(z_\circ\) lies inside the unit disk.

The constant \(C\) can be written in a compact form
\[
C=\frac{2\imath(1-e^{2j})(\cosh k+\cos m)}{\pi(1+\imath e^{j})^{2}},
\qquad
j=\tfrac{s-t}{2},\quad k=\tfrac{s+t}{2}.
\]
%Alternatively, in terms of hyperbolic functions of \(j/2\), we have
%\[
%C=-\frac{(\cos m+\cosh k)
%(\cosh\tfrac{j}{2}-\imath\sinh\tfrac{j}{2})\sinh j}
%{\pi(\cosh\tfrac{j}{2}+\imath\sinh\tfrac{j}{2})^{3}}.
%\]
%When \(z\) approaches \(1\), the function \(K(z)\) behaves as
%\[
%K(z)
%=\frac{|1-z_\circ|^{2}}{2}\,
%\frac{(\cos m+\cosh k)\sinh j}{2\pi}\,
%\frac{1}{1-z}
%+O((z-1)).
%\]
%This expansion is often useful for analyzing the behavior of \(K\) near the boundary of the unit disk.

Finally, the expression for \(K(z)\) can be summarized as
\begin{equation}\label{KKK}
K(z)
=C\,\frac{(z-z_\circ)(1-z\,\overline{z_\circ})}
{(1-z^{2})(e^{2\imath p}-z^{2})},
\qquad
e^{\imath p}=\frac{(\imath+e^{j})^{2}}{(\imath-e^{j})^{2}}.
\end{equation}
%where the constant \(J\) is given by
%\[
%C
%=-\frac{(\cos m+\cosh k)
%(\cosh\frac{j}{2}-\imath\sinh\frac{j}{2})\sinh j}
%{\pi(\cosh\frac{j}{2}+\imath\sinh\frac{j}{2})^{3}},
%\qquad
%j=\tfrac{s-t}{2},\quad
%k=\tfrac{s+t}{2},\quad
%s>t>0.
%\]
Moreover, for all real \(s,t,m\), one has the identity
\begin{equation}\label{KeyI}
\frac{C}{e^{2\imath p}-1}
=-\frac{\imath\,\cosh\!\big(\tfrac{s-t}{2}\big)
\big(\cos m+\cosh\!\big(\tfrac{s+t}{2}\big)\big)}{4\pi}.
\end{equation}

\paragraph{The principal part notation.}
For a meromorphic \(h\) at \(z_0\), we write \(\mathcal{P}_{z=z_0} h\) for the principal part (the \( (z-z_0)^{-1}\)-term).

\subsection*{The expansion near \(z=1\)}
Define the principal part
\begin{equation}\label{eq:k-def}
H(z):= \mathcal{P}_{z=1} K(z).
\end{equation}
A direct Taylor expansion of the denominator in \eqref{KKK} gives
\(1-z^2= -2(z-1)+O\bigl((z-1)^2\bigr)\) and \(e^{2\imath p}-z^2=(e^{2\imath p}-1)+O(z-1)\).
Evaluating the numerator at \(z=1\) yields \((1-z_\circ)(1-\overline{z_\circ})=|1-z_\circ|^2\).
Hence
\begin{equation}\label{eq:pp-at-1}
\mathcal{P}_{z=1} K(z)=
-\frac{C}{2\,(e^{2\imath p}-1)}\,\frac{|1-z_\circ|^2}{z-1}.
\end{equation}
Using \eqref{KeyI} we rewrite this in two equivalent (and useful) ways:
\begin{align}
\mathcal{P}_{z=1} K(z)
&=\;\imath\,\frac{\cosh\!\bigl(\tfrac{s-t}{2}\bigr)\,\bigl(\cos m+\cosh\!\bigl(\tfrac{s+t}{2}\bigr)\bigr)}{4\pi}\;
\frac{|1-z_\circ|^2}{z-1}, \label{eq:pp-at-1-formA}\\[4pt]
&=\;\frac{\imath\,\cosh\!\bigl(\tfrac{s-t}{2}\bigr)\,\bigl(\cos m+\cosh\!\bigl(\tfrac{s+t}{2}\bigr)\bigr)}{4\pi\,\sin p}\;
\frac{|1-z_\circ|^{2}}{z-1}.
\label{eq:pp-at-1-formB}
\end{align}

\subsection*{The principal parts at all four poles}
\begin{proposition}[The principal parts of \(K\)]\label{prop:pp}
Let \(K\) be given by \eqref{KKK}. Then, as \(z\) tends to each pole,
\begin{align}
z\to 1:\quad
K(z) &= \phantom{-}\frac{\imath\,\cosh\!\bigl(\tfrac{s-t}{2}\bigr)\,\bigl(\cos m+\cosh\!\bigl(\tfrac{s+t}{2}\bigr)\bigr)}{4\pi\,\sin p}\;
\frac{|1-z_\circ|^{2}}{z-1}
+O(1), \label{eq:p1}\\[6pt]
z\to e^{\imath p}:\quad
K(z) &=-\frac{\imath\,\cosh\!\bigl(\tfrac{s-t}{2}\bigr)\,\bigl(\cos m+\cosh\!\bigl(\tfrac{s+t}{2}\bigr)\bigr)}{4\pi\,\sin p}\;
\frac{|1-z_\circ e^{-\imath p}|^{2}}{z-e^{\imath p}}
+O(1), \label{eq:p2}\\[6pt]
z\to -1:\quad
K(z) &= \phantom{-}\frac{\imath\,\cosh\!\bigl(\tfrac{s-t}{2}\bigr)\,\bigl(\cos m+\cosh\!\bigl(\tfrac{s+t}{2}\bigr)\bigr)}{4\pi\,\sin p}\;
\frac{|1+z_\circ|^{2}}{z+1}
+O(1), \label{eq:p3}\\[6pt]
z\to -e^{\imath p}:\quad
K(z) &=-\frac{\imath\,\cosh\!\bigl(\tfrac{s-t}{2}\bigr)\,\bigl(\cos m+\cosh\!\bigl(\tfrac{s+t}{2}\bigr)\bigr)}{4\pi\,\sin p}\;
\frac{|1+z_\circ e^{-\imath p}|^{2}}{z+e^{\imath p}}
+O(1). \label{eq:p4}
\end{align}
\end{proposition}

\begin{proof}
We show the computation at \(z=e^{\imath p}\); the other cases are analogous.

Write \(z=e^{\imath p}+w\). Then
\[
e^{2\imath p}-z^2 = e^{2\imath p}-(e^{\imath p}+w)^2
= -2 e^{\imath p}\,w + O(w^2),\qquad
1-z^2 = 1-e^{2\imath p} + O(w).
\]
Thus
\begin{equation*}
\frac{1}{(1-z^2)(e^{2\imath p}-z^2)}
= \frac{1}{(1-e^{2\imath p})}\,\frac{1}{-2 e^{\imath p}\,w} + O(1)
= \frac{1}{2 e^{\imath p}(e^{2\imath p}-1)}\,\frac{1}{w} + O(1).
\end{equation*}
For the numerator, evaluate at \(z=e^{\imath p}\):
\begin{equation*}
(z-z_\circ)\bigl(1-z\,\overline{z_\circ}\bigr)\Big|_{z=e^{\imath p}}
=(e^{\imath p}-z_\circ)\bigl(1-e^{\imath p}\overline{z_\circ}\bigr)
= e^{\imath p}\,|1-z_\circ e^{-\imath p}|^2.
\end{equation*}
Therefore
\begin{equation*}
K(z)= C\cdot \frac{e^{\imath p}\,|1-z_\circ e^{-\imath p}|^2}{2 e^{\imath p}(e^{2\imath p}-1)}
\cdot \frac{1}{z-e^{\imath p}} + O(1)
= \frac{C}{2(e^{2\imath p}-1)}\,\frac{|1-z_\circ e^{-\imath p}|^2}{z-e^{\imath p}} + O(1).
\end{equation*}
Using \eqref{KeyI} we get \eqref{eq:p2}. The remaining cases follow by the same Taylor expansion:
near \(z=1\), \(1-z^2=-2(z-1)+\cdots\);
near \(z=-1\), \(1-z^2=2(z+1)+\cdots\);
near \(z=-e^{\imath p}\), use
\(e^{2\imath p}-z^2 = 2 e^{\imath p}(z+e^{\imath p})+\cdots\).
In each case the numerator collapses to the corresponding squared modulus,
e.g. \(|1\mp z_\circ|^2\) or \(|1\mp z_\circ e^{-\imath p}|^2\).
\end{proof}

\subsection*{Asymptotics for the antiderivative}
Let \(H(z)=\displaystyle\int_{0}^{z} K(\zeta)\,\mathrm{d}\zeta\) along any path avoiding the poles.
\begin{lemma}\label{lem:log}
Suppose \(K(z)=\dfrac{A}{z-z_0}+O(1)\) with \(A=\imath C\) and \(C\in\mathbb{R}\). Then
\[
\Im H(z)= C\,\log|z-z_0| + O(1),\qquad z\to z_0.
\]
\end{lemma}
\begin{proof}
Write \(g(z)=\dfrac{A}{z-z_0}+h(z)\) with \(h(z)=O(1)\). Then
\(
F(z)=A\log(z-z_0)+\int h(\zeta)\,\mathrm{d}\zeta
\)
locally, for a fixed branch of \(\log\). Since \(\int h=O(|z-z_0|)\),
\(
\Im F(z)=\Im\!\big(A\log(z-z_0)\big)+O(1)=C\log|z-z_0|+O(1).
\)
\end{proof}

\begin{corollary}[Logarithmic laws for \(\Im H\)]
With the constants
$$
\Lambda=\dfrac{\cosh\!\bigl(\tfrac{s-t}{2}\bigr)\bigl(\cos m+\cosh\!\bigl(\tfrac{s+t}{2}\bigr)\bigr)}{4\pi\,\sin p}
$$
and
$$
C_1=\Lambda|1-z_\circ|^{2},\;
C_2=\Lambda|1-z_\circ e^{-\imath p}|^{2},\;
C_3=\Lambda|1+z_\circ|^{2},\;
C_4=\Lambda|1+z_\circ e^{-\imath p}|^{2},
$$
we have, as \(z\) approaches each pole,
\begin{align*}
z\to 1:&\quad \Im H(z)= +\,C_1\,\log|z-1|+O(1),\\
z\to e^{\imath p}:&\quad \Im H(z)= -\,C_2\,\log|z-e^{\imath p}|+O(1),\\
z\to -1:&\quad \Im H(z)= +\,C_3\,\log|z+1|+O(1),\\
z\to -e^{\imath p}:&\quad \Im H(z)= -\,C_4\,\log|z+e^{\imath p}|+O(1).
\end{align*}
In particular $$\lim_{z\to \pm 1}\Im H(z)=+\infty $$ $$\lim_{z\to \pm e^{\imath p}}\Im H(z)=-\infty .$$ 
\end{corollary}

\section{A proof of Theorem~\ref{etreta}}
We start with the reminder that the harmonic center was defined in the Introduction.
\begin{lemma}
Let $Q=Q(b_1,b_2,b_3,b_4)$ be a Pitot quadrilateral and denote its harmonic center by $c_0$. 
Define \begin{align}
w &= u+\imath v
    = \frac{2b_4 - b_1 - b_3}{b_3 - b_1}
    = \sin m \cosh s + \imath\,\cos m \sinh s,
    \label{tsm-w} \\[4pt]
z &= x+\imath y
    = \frac{2b_2 - b_1 - b_3}{b_3 - b_1}
    = \sin m \cosh t + \imath\,\cos m \sinh t.
    \label{tsm-z}
\end{align}
where $2\sin m=|z+1|-|z-1|=|w+1|-|w-1|$. Let $j=(s-t)/2$ and $k=(s+t)/2$.
Then the curvature of the minimal graph $\Sigma$ at the point $\xi$, above the harmonic center of $Q$ is given by \begin{equation}\label{desire}\mathcal{K}^0(m)=-\frac{1}{4} \pi ^2 \cos^2 m  \coth^2 j\,\mathrm{sech}^4 k.\end{equation}
Moreover, 
\[
|\mathcal {K}^{Q}(m)| \;\le\; |\mathcal{K}^{Q}(0)|
   = \frac{\pi^{2}}{4}\,\coth^{2} j\,\mathrm{sech}^{4} k .
\]
\end{lemma}
\begin{proof}
We use the formula for the Gaussian curvature
\begin{equation}\label{kurva}
\mathcal{K} = -\frac{4\,|{\bf q}'(z)|^{2}}{|{\bf p}(z)|^{2}\left(1+|{\bf q}(z)|^{2}\right)^{4}}.
\end{equation} By carrying out a lengthy but  straightforward computation and using the previously established constants, we obtain
\begin{equation}\label{qn}
{\bf q}(0)
=
-i\,\operatorname{sech}\!\left(\tfrac{1}{2}(k - i m)\right)
    \sinh\!\left(\tfrac{1}{2}(k + i m)\right),
\end{equation}
\[
{\bf q}'(0)
=
\frac{(1 + i e^{j})\,\cos m}{i + e^{j}}\,
\operatorname{sech}^{2}\!\left(\tfrac{1}{2}(k - i m)\right),
\]
and 
\[
h'(0)
=
-\frac{2 i\,(e^{2j}-1)\,\bigl(1+\cosh(k - i m)\bigr)}
       {(i+e^{j})^{2}\,\pi}.
\]
By substitution in \eqref{kurva} we get \eqref{desire}.
\end{proof}
We have already established a harmonic diffeomorphism of the unit disk onto a Pitot quadrilateral $Q=Q(b_1,b_2,b_3,b_4)$, whose dilatation is a square of a M\"obius transformation $\phi$. The harmonic center $c_0$  of $Q$ is the unique point $c_0=f(0)\in Q$. Let $\xi$ be the point of the corresponding minimal graph $\Sigma_Q$, determined in Theorem~\ref{teo2}, above $c_0$. In our situation  \[c_0=
\frac{1}{2\pi}\int_{0}^{2\pi} F(\theta)\,d\theta
=
\frac{1}{2\pi}  (p b_2+(\pi - p) b_3+p b_4 +(\pi - p)b_1)=\frac{p}{\pi}(z + w).
\]

Using a formula from \cite{kalaj2025gaussian}, we have
\begin{equation}\label{fexpli}
{f}_{uv}
= -\frac{2\,\Re\!\left[{h'(0)}\,(1-{{\bf q}(0)}^{4})\,\overline{{\bf q}'(0)}\right]}
       {|{h'(0)}|^{2}\,(1-|{{\bf q}(0)}|^{2})^{3}\,(1+|{{\bf q}(0)}|^{2})}.
\end{equation}
Evaluated at the point $c_0$, this becomes
\[
{f}_{uv}(c_0)
    = \frac{1}{4}\,\pi\,\coth(j)\,\sec(m).
\]

%Moreover, the complex numerator in \eqref{fexpli} satisfies
%\[
%-\frac{2\,{h'}(1-{q}^{4})\,\overline{{q}'}}%
 %     {|{h'}|^{2}(1-|{q}|^{2})^{3}(1+|{q}|^{2})}
%=
%\frac{1}{4}\,\pi\,\coth(j)\,\sec(m)\,
%\bigl(1 - i\,\sin(m)\,\sinh(k)\bigr).
%\]
 \begin{figure}
\centering
\includegraphics[scale=0.6]{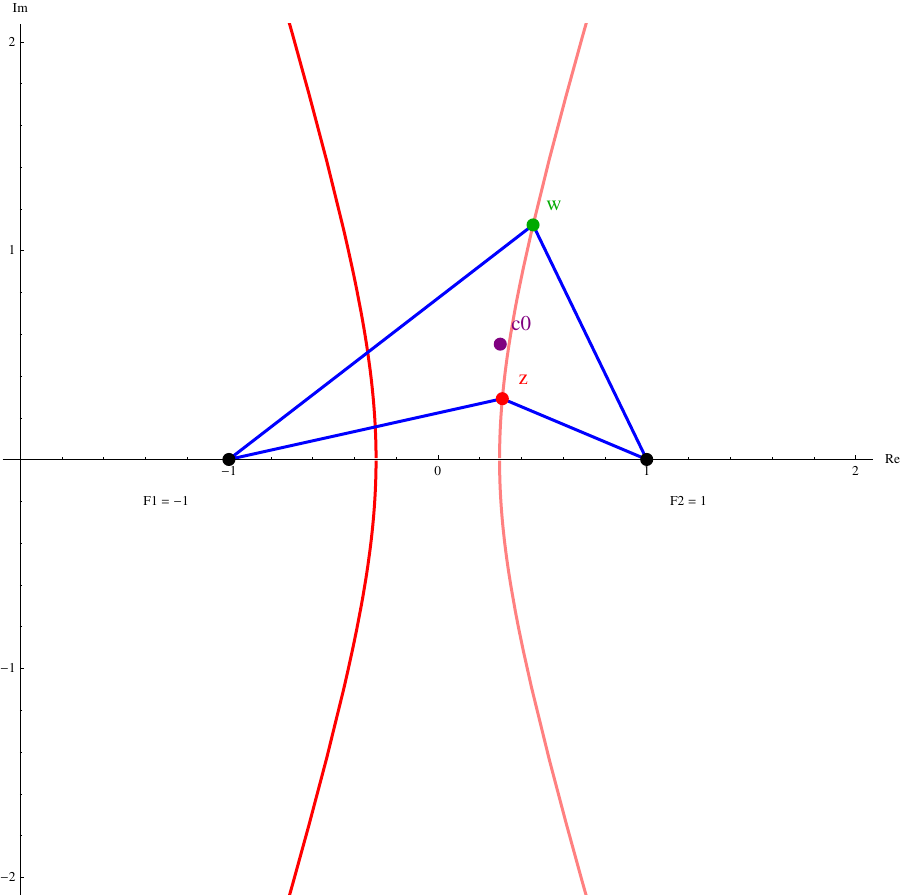}
\caption{The quadrilateral whose opposite vertices are the foci $-1$ and $1$,
and whose remaining vertices lie on the same branch of the hyperbola $|\zeta+1|-|\zeta-1|=2\sin m$, i.e. of 
\(
x^{2}\csc^{2} m - y^{2}\sec^{2} m = 1.
\)}
\label{quad}
\end{figure}
\begin{remark}
Consider a rotation of the quadrilateral $Q$ by an angle $\alpha$.
The corresponding harmonic mapping is
\[
f_\alpha(z)=e^{i\alpha} f(z),
\]
so
\[
{{\bf p}}_\alpha(0)=\partial f_\alpha(0)=e^{i\alpha}{{\bf p}}(0),
\qquad
{{\bf q}}_\alpha(z)=e^{-i\alpha}{{\bf q}}(z).
\]

A straightforward computation then yields
\[
\begin{split}&-\frac{
2\,\Re\!\left[
h'_\alpha(1-{{\bf q}}_\alpha^{4})\,
\overline{{{\bf q}}_\alpha'}
\right]
}{
|h'_\alpha|^{2}
(1-|{\bf q}_\alpha|^{2})^{3}
(1+|{\bf q}_\alpha|^{2})
}
\\&=
\frac{\pi (1+e^{2j})\bigl(e^{2k}\cos(2\alpha-m)+\cos(2\alpha+m)\bigr)}{16 e^{j+k}\sinh j \cos^2 m \cosh k}\,\,
\,
\,.\end{split}
\]

Thus, by choosing
\[
\alpha
=
\pm\arccos\!\left(
\pm
\sqrt{
\frac12
\pm
\frac{\sin(m)\,\sinh(k)}%
     {\sqrt{2}\,\sqrt{\cos(2m)+\cosh(2k)}}
}
\right),
\]
the mixed derivative of the rotated map $\mathbf{f}_\alpha$ at
\[
c_\alpha = e^{i\alpha} c_0
\]
vanishes.
We denote the rotated quadrilateral by
\[
Q_\alpha = e^{i\alpha} Q.
\]
We translate it by $c_\alpha$ and get $Q_0 = Q_\alpha-c_\alpha$. Then $0$ is the harmonic center of $Q_0$.
Now, the new graph over $Q_0$ satisfies the condition $\mathbf{f}_{uv}(0,0)=0$, which means that at the origin, the local coordinate system formed by the $u$ and $v$ axes is already aligned with the principal directions. This simplifies calculations and provides a clear geometric interpretation: the surface is bending maximally in one coordinate direction and minimally (with opposite sign curvature) in the other, and these directions are aligned with the chosen $u$ and $v$ axes.  We say that such a graph is \emph{bending in coordinate directions at zero}. Notice that every minimal graph can be transformed into such a graph by an appropriate rotation around the $z$-axis.
\end{remark} Then the stereographic projection of ${\bf q}(0)$, from \eqref{qn} is
\[
N_\xi =\frac{1}{1+|{\bf q}(0)|^2}\left(\Re {\bf q}(0), \Im {\bf q}(0), (1-|{\bf q}(0)|^2)\right)=
\left(
\sin m,
-\cos m \,\tanh k,
\cos m \,\operatorname{sech} k
\right),
\] 
which is the unit normal of $\Sigma^\diamond$ at the point $\xi=(c_0, \mathbf{f}^\diamond(c_0))$.

Now Theorem~\ref{etreta} is a direct consequence of the following theorem.
\begin{theorem}\label{kater}
Assume that $\Sigma=\{(u,v, \mathbf{f}(u,v)): (u,v)\in Q\}$ is any bounded minimal graph above $Q$. Then for $\xi = (c_0, \mathbf{f}(c_0))$ we have $$\left|\mathcal{K}(\xi)\right| <  \frac{\pi ^2 \cos^2 m  \coth^2 j\,\mathrm{sech}^4 k}{|b_1-b_3|^2}$$ provided that the unit normal of $\Sigma$ at $\xi$ is $$N_\xi=\left(
\sin m,
-\cos m \,\tanh k,
\cos m \,\operatorname{sech} k
\right)$$ and $\mathbf{f}_{u,v}(c_0)=\frac{1}{4}\,\pi\,\coth(j)\,\sec(m)$. Here, the constants $k=(s+t)/2$, $j=(s-t)/2$, $t$, $s$ and $m$ depend on $b_1,b_2,b_3$ and $b_4$ and can be explicitly determined from \eqref{tsm-z} and \eqref{tsm-w}. The result is sharp.
\end{theorem}
 \begin{figure}
\centering

\includegraphics{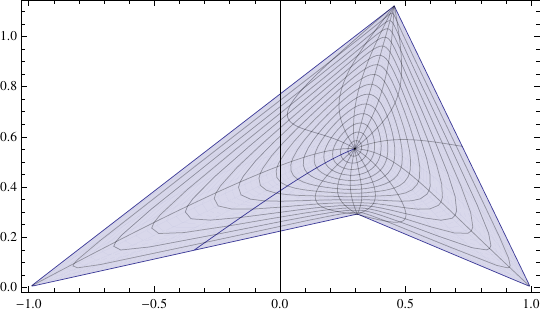}
\caption{The concave quadrilateral $Q(-1,z,1,w)$ for $m=0.3$, $t=0.3$, $s=1$, $c_0= 0.299+ 0.552\imath$.}
\label{quad}
\end{figure}
 \begin{figure}
\centering
\includegraphics[scale=0.6]{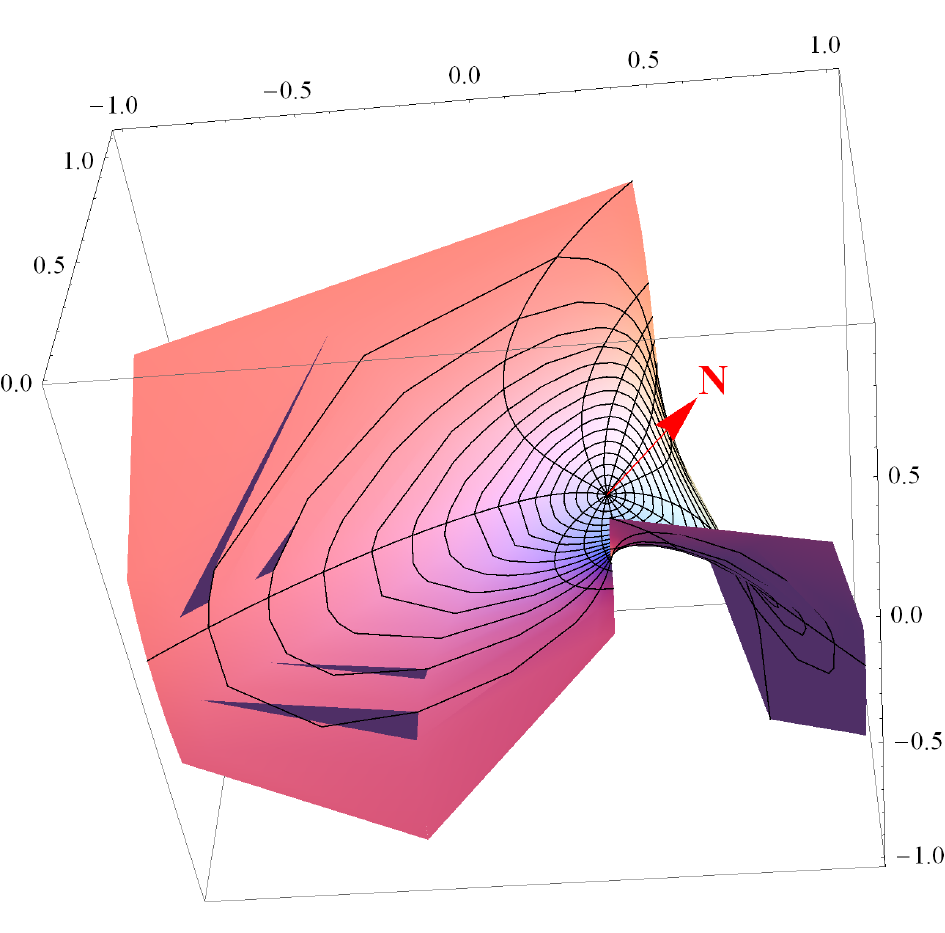}
\caption{The Scherk type surface over the concave $Q$ with its unit normal $N$ at the center.}
\label{quad}
\end{figure}
 \begin{figure}
\centering
\includegraphics{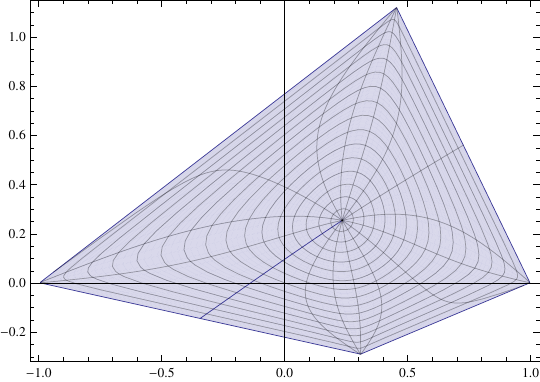}
\caption{The convex quadrilateral $Q(-1,z,1,w)$ for $m=0.3$, $t=-0.3$, $s=1$, $c_0= 0.234 +0.255 \imath$.}
\label{quad}
\end{figure}
 \begin{figure}
\centering
\includegraphics[scale=0.6]{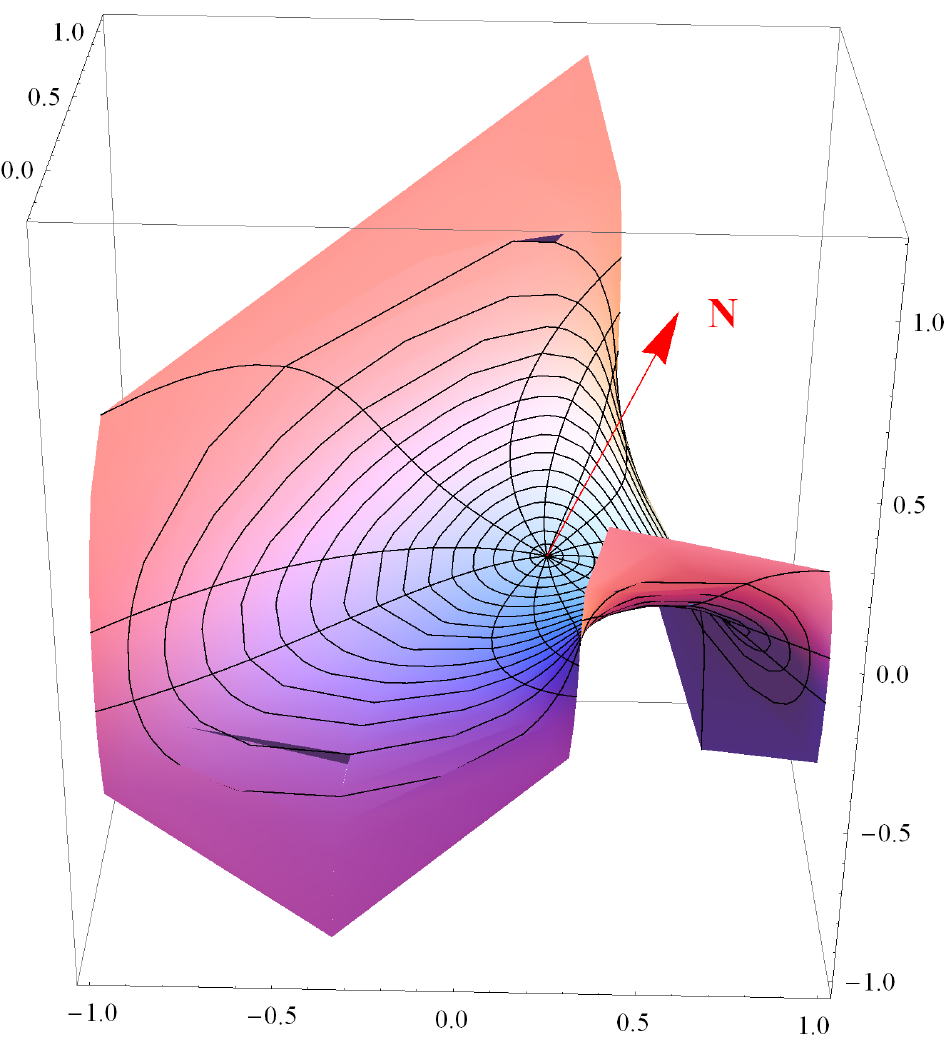}
\caption{The Scherk type surface over the convex $Q$ with its unit normal $N$ at the center.}
\label{quad}
\end{figure}
\begin{proof} As before, we can assume that $b_1=-1$ and $b_2=1$.
Assume the opposite:
$$\left|\mathcal{K}(\xi)\right| \ge   \left|\mathcal{K}^0(\xi)\right|.$$ 
Let $Q = Q(b_1,b_2,b_3,b_4)\subset D\subset\mathbb{R}^2$ be a given quadrilateral. We are given a Scherk type minimal graph
$z\mapsto f_1(z)$ defined over $Q$ and can assume, without loss of generality, that the second is a minimal graph $z\mapsto f_2(z)$
defined over a domain $D$ containing $Q$. Assume that the two graphs
intersect at the point
\[
\xi=(c_0,f_1(c_0)),
\]
and that at $c_0$ the corresponding minimal surfaces $G_{f_1}$ and
$G_{f_2}$ have the same Gaussian  curvature and the same mixed derivatives.

\medskip

We first reduce the situation by a similarity of space. Translate
horizontally so that $c_0$ is moved to the origin, and vertically so
that $f_1(0)=f_2(0)=0$.  In
these new coordinates both surfaces are still graphs over the
$xy$–plane, say $z=\phi_1(x,y)$ and $z=\phi_2(x,y)$, defined at least
in a neighborhood of the origin, and we still have
\[
\phi_1(0,0)=\phi_2(0,0)=0.
\]
Finally, as in Finn and Osserman argument (\cite{5} and  \cite{FINNOSSERMAN}), we apply a homothety of $\mathbb{R}^3$
with the factor $\lambda\ge 1$, so that the two surfaces have equal Gaussian curvatures at the
origin. Observe that such a $\lambda$ expands the domain $D$ and increases the Gaussian curvature. Thus, we may assume from now on that
\[
{\mathcal K}_{G_{\phi_1}}(0,0)={\mathcal K}_{G_{\phi_2}}(0,0).
\]

By assumption, the two graphs have the same unit normals at the origin
and the same mixed vanishing derivatives there. Together with the minimal surface
equation, this implies that their second fundamental forms at the
origin coincide (see Lemma~\ref{lemab})). Hence
\[
\phi_1(0,0)=\phi_2(0,0),\qquad \nabla\phi_1(0,0)=\nabla\phi_2(0,0),
\qquad D^2\phi_1(0,0)=D^2\phi_2(0,0).
\]

\medskip

Consider now the difference
\[
\phi(x,y)=\phi_1(x,y)-\phi_2(x,y).
\]
It is defined in a neighborhood of the origin, and by the previous
remarks, we have
\[
\phi(0,0)=0,\qquad \nabla\phi(0,0)=0,\qquad D^2\phi(0,0)=0.
\]
Moreover, both $\phi_1$ and $\phi_2$ satisfy the minimal surface
equation, thus $\phi$ satisfies a linear uniformly elliptic equation
with analytic coefficients. We are therefore in the situation of
\cite[Lemma~1]{FINNOSSERMAN} with $n=2$ (after a preliminary rotation already been made).
Hence, there exists a homeomorphism of a neighborhood of the origin in
the $xy$–plane onto a neighborhood of the origin in the $\zeta$–plane,
such that
\[
\phi(x,y)=\Re\{\zeta^N\},\qquad N\ge 3,
\]
in some neighborhood of the origin. In particular, the level set
$\phi=0$ near the origin consists of $N$ analytic arcs intersecting
only at the origin, and they divide a small neighborhood of the origin
into $2N$ sectors in which $\phi$ is alternately positive and negative.

\medskip

Next, we use the special boundary behavior of the Scherk type graph
$\phi_1$ over $Q$. By assumption, $\phi_1$ has the usual Scherk
singularities along the sides of $Q$: on two opposite sides, it tends
to $+\infty$, on the other two sides, it tends to $-\infty$, whereas
$\phi_2$ extends smoothly across $Q$ (since it is defined on $D$
containing $Q$). It follows that along each side of $Q$ the function
$\phi=\phi_1-\phi_2$ tends to either $+\infty$ or $-\infty$ while the
other graph remains finite. Consequently, any component of the open set
$\{\phi\neq 0\}$ that meets an interior point of a side of $Q$ must
contain  that entire side. Thus, at most four components of
$\{\phi\neq 0\}$ can have the boundary that intersects the sides of $Q$.
In particular, the set $\{\phi\neq 0\}$ has at most four components
whose boundaries contain interior points of the sides of $Q$.

On the other hand, near the origin, the local description
$\phi=\Re\{\zeta^N\}$ with $N\ge 3$ shows that $\{\phi\neq 0\}$ has at
least six components accumulating at the origin, corresponding to the
$2N$ alternating sectors. If all these components had boundaries that
meet only interior points of $Q$, then by joining suitable sectors we
could construct a component whose boundary is entirely interior to
$Q$, contradicting the maximum principle for solutions of the
underlying linear elliptic equation (compare \cite[p.~322]{8}). Thus, at
least one of these components must have the boundary consisting only of
interior points of $Q$ and one or more of its vertices.

Finally, we apply the maximum principle once more: as in
\cite[Theorem~II.1]{5}, a nontrivial solution $\phi$ of the linearized
equation cannot attain both positive and negative values in a domain
whose boundary consists only of interior points of $D$ and isolated
vertices. Hence, we would again be forced to conclude that
$\phi\equiv 0$ in a neighborhood of the origin, which contradicts the
representation $\phi=\Re\{\zeta^N\}$ with $N\ge 3$.

This contradiction shows that our initial assumption was impossible.
Therefore, there cannot exist a Scherk type minimal graph $f_1$ over
$Q$ and a minimal graph $f_2$ over a domain $D\supset Q$ such that
$G_{f_1}$ and $G_{f_2}$ have at $0$ the same Gaussian curvature, the
same unit normals and the same mixed second derivatives.

\end{proof}

\begin{lemma}\label{lemab}
Let $f^\ast,f^\diamond : U\subset\mathbb{R}^2\to\mathbb{R}$ be minimal graphs.
Assume that at $(0,0)$:
\begin{enumerate}
\item [(i)] their unit normals coincide;
\item [(ii)] their mixed derivatives coincide: $f^\ast_{uv}=f^\diamond_{uv}$;
\item [(iii)] their Gaussian curvatures coincide.
\end{enumerate}
Then at $(0,0)$ either
\[
D^2 f^\ast = D^2 f^\diamond
\qquad\text{or}\qquad
D^2 f^\ast = -\,D^2 f^\diamond,
\]
hence $D^2(f^\ast\pm f^\diamond)(0,0)=0$.
\end{lemma}

\begin{proof}
The equality of the normals gives the equality of the gradients:
\[
\nabla f^\ast(0,0)=\nabla f^\diamond(0,0)=({ \alpha,\beta}).
\]
Subtract their common linear part ${ \alpha}u+{ \beta}v$.
This does not change the second derivatives. The new graphs also have the same Gaussian curvatures at the point $(0,0)$. Still,  denote the new functions  by  the old symbols $f^\ast$ and $f^\diamond$. Then, 
\[
\nabla f^\ast(0,0)=\nabla f^\diamond(0,0)=(0,0).
\]

Write
\[
a_1=f^\ast_{uu},\quad a_2=f^\ast_{vv},\quad c=f^\ast_{uv}=f^\diamond_{uv},
\]
\[
A_1=f^\diamond_{uu},\qquad A_2=f^\diamond_{vv}.
\]

At a point where $f_u=f_v=0$, the minimality gives
\[
f_{uu}+f_{vv}=0.
\]
Hence,
\[
a_1+a_2=0,\qquad A_1+A_2=0,
\]
thus, $a_2=-a_1$ and $A_2=-A_1$.

At a point with the zero gradient, the Gaussian curvature is
\[
{\mathcal K}=f_{uu}f_{vv}-f_{uv}^2.
\]
The equalities of the curvatures and of the mixed derivatives give
\[
a_1a_2-c^2 = A_1A_2-c^2,
\]
hence,
\[
a_1a_2 = A_1A_2.
\]
Using $a_2=-a_1$ and $A_2=-A_1$,
\[
-a_1^2=-A_1^2,
\]
so $A_1=\pm a_1$, and thus $A_2=\pm a_2$ with the same sign.

Either $(A_1,A_2)=(a_1,a_2)$ or $(A_1,A_2)=(-a_1,-a_2)$, i.e.
\[
D^2 f^\diamond = D^2 f^\ast
\quad\text{or}\quad
D^2 f^\diamond = -\,D^2 f^\ast.
\]
Thus, $D^2(f^\ast\pm f^\diamond)(0,0)=0$.
\end{proof}

\section*{Acknowledgements}

The authors would like to express their sincere gratitude to Professors Daoud Bshouty and Abdallah Lyzzaik for many fruitful and insightful discussions on the subject, which greatly contributed to the development of this work. Their expertise and guidance have been invaluable throughout the preparation of this paper.
V.D. thanks the Simons Foundation grant no.~854861, the Serbian Ministry of Science, Technological Development and Innovation and the Science Fund of Serbia grant IntegraRS. 
D.K. gratefully acknowledges financial support from the grant “Mathematical Analysis, Optimisation and Machine Learning.”

\end{document}